\documentclass[smallextended]{svjour3} 
\usepackage[utf8]{inputenc}
\usepackage[T1]{fontenc}
\usepackage{algorithm}
\usepackage[noend]{algpseudocode}
\usepackage{float}
\usepackage{textcomp}
\usepackage{mathrsfs}
\usepackage{mathtools}
\usepackage{amsmath}

\usepackage{amsthm}
\usepackage{amssymb}
\usepackage{stmaryrd}
\usepackage{esint}

\newcommand{\sgoal}{\textsc{SGoal}}
\newcommand{\sgoals}{\textsc{SGoal}s}

\providecommand{\tabularnewline}{\\}
\floatstyle{ruled}
\newfloat{algorithm}{tbp}{loa}
\providecommand{\algorithmname}{Algorithm}
\floatname{algorithm}{\protect\algorithmname}

\theoremstyle{plain}
\newtheorem{thm}{\protect\theoremname}
\theoremstyle{definition}

\theoremstyle{plain}
\newtheorem{lem}[thm]{\protect\lemmaname}
\theoremstyle{plain}
\newtheorem{prop}[thm]{\protect\propositionname}
\theoremstyle{remark}

\theoremstyle{plain}
\newtheorem{cor}[thm]{\protect\corollaryname}
\providecommand{\corollaryname}{Corollary}
\providecommand{\definitionname}{Definition}
\providecommand{\lemmaname}{Lemma}
\providecommand{\propositionname}{Proposition}
\providecommand{\remarkname}{Remark}
\providecommand{\theoremname}{Theorem}

\usepackage{natbib}
\usepackage{graphicx}

\newcommand{\myreplace}{\mathbin{\smash{%
\raisebox{0.35ex}{%
            $\underset{\raisebox{0.5ex}{$\smash ,$}}{\smash+}$%
            }%
        }%
    }%
}

\newcommand{\eqtriangle}{\mathbin{\smash{%
\raisebox{0.35ex}{%
            $\underset{\raisebox{0.5ex}{$\smash -$}}{\smash\triangleleft}$%
            }%
        }%
    }%
}

\begin{document}

\title{Non-Stationary Stochastic Global Optimization Algorithms}

\author{Jonatan Gomez
\and
Carlos Rivera
}

\institute{  J. G\'omez \at Computer Systems Engineering \\
                Engineering School \\ Universidad Nacional de Colombia \\
              \email{jgomezpe@unal.edu.co}           
              \\
              \and
              C. Rivera \at Computer Systems Engineering \\
              Engineering School \\ Universidad Nacional de Colombia \\
              \email{cariveram@unal.edu.co} 
}

\maketitle

\begin{abstract}
     \cite{GOMEZ201953} proposes a formal and systematic approach for characterizing stochastic global optimization algorithms. Using it, Gomez formalizes algorithms with a fixed next-population stochastic method, i.e., algorithms defined as stationary Markov processes. These are the cases of standard versions of hill-climbing, parallel hill-climbing, generational genetic, steady-state genetic, and differential evolution algorithms. This paper continues such a systematic formal approach. First, we generalize the sufficient conditions convergence lemma from stationary to non-stationary Markov processes. Second, we develop Markov kernels for some selection schemes. Finally, we formalize both simulated-annealing and evolutionary-strategies using the systematic formal approach.
    \keywords{ Evolutionary Algorithms \and Non-stationary Markov Kernel \and Convergence Analysis \and Evolutionary Strategies \and Simulated Annealing \and Selection Mechanism}
    \subclass{MSC 68T20 \and MSC 65K10}
\end{abstract}

\section{Introduction}
    This section provides a brief introduction to the systematic formalization proposed by \cite{GOMEZ201953}. Such systematic formalization of stochastic global optimization algorithms (\textbf{\sgoals} in short), is carried on Markov kernels terms. Gomez can formalize {\sgoals} with fixed \textsc{NextPop} stochastic method, i.e., {\sgoals} that can be characterized as stationary Markov processes. That is the case of the hill-climbing  (\cite{Russell:2009:AIM:1671238}), the parallel hill-climbing, the generational genetic  (\cite{DeJong75,Holland75,Mitchell96}), the steady-state genetic (\cite{GOLDBERG199169}), and the differential evolution (\cite{DasDE2011,Storn:1997:DEN:596061.596146}) algorithms. However, {\sgoals} such as the Simulated Annealing (\cite{Kirkpatrick671}), Evolutionary Strategies (\cite{Beyer2002}), or any algorithm using parameter control/adaptation techniques (\cite{Eiben99}) cannot be characterized as stationary Markov processes. 

    \subsection{Stochastic Global Optimization}
        The problem of finding a point $x^{*}\in\varOmega\subseteq\varPhi$ where a function $f\vcentcolon\varPhi\rightarrow\mathbb{R}$ reaches its best/optimal value ($f^{*}$), is considered as a global optimization problem, see equation \ref{eq:GlobalMin}. Here, $\varPhi$ is the solution space, $\varOmega$ is the feasible region, $x^{*}$ is the optimizer, $f$ is the objective function, and $\vartriangleleft$ is the optimization relation: $\vartriangleleft$ is $<$ if minimizing and it is $>$ if maximizing.
        \begin{equation}
            \label{eq:GlobalMin}
            optimize\left(f:\varPhi\rightarrow\mathbb{R}\right)=x^*\in\varOmega\subseteq\varPhi\mid\left(\forall y\in\varOmega\right)\left(f\left(x^*\right) \trianglelefteq f\left(y\right)\right)
        \end{equation}

       A stochastic global optimization algorithm ({\sgoal}) iteratively generates a (possibly) better population of candidate solutions using a stochastic operation, see algorithm \ref{Alg:SGoal}. Here, $\textsc{InitPop}:\mathbb{N}\rightarrow\varOmega^{n}$ initializes a population $P$ having size $n$, \textit{\textsc{$\textsc{NextPop}:\varOmega^{n}\rightarrow\varOmega^{n}$}} stochastically generates a new population from the current one, \textsc{$\textsc{Best}:\varOmega^{n}\rightarrow\varOmega$} obtains the fittest individual (see equation \ref{eq:Best}), and \textit{\textsc{$\textsc{End}:\varOmega^{n}\times\mathbb{N}\rightarrow Bool$}} is a stopping condition.  Notice that if there is a Markov kernel characterizing the \textsc{NextPop} method, the stochastic sequence $\left(P_{t}\vcentcolon t\geq0\right)$ becomes a Markov process.

        \begin{algorithm}
            \caption{\label{Alg:SGoal}Stochastic Global Optimization Algorithm - \sgoal.}
            \begin{flushleft}\sgoal($n$)\end{flushleft} 
            \begin{algorithmic}[1]
                \State $t=0$
                \State $P_{0}$ = \textsc{InitPop}($n$) 
                \While{\textsc{\textlnot End}($P_{t}$ , $t$)}
                    \State $P_{t+1}$ = \textsc{NextPop}( $P_{t}$ )
                    \State $t=t+1$
                \EndWhile
                \State \textbf{return} \textsc{best}($P_{t}$)
            \end{algorithmic}
        \end{algorithm}

        \begin{equation}
            \textsc{Best}\left(x\right)=x_{i}\mid\forall_{k=1}^{n}f\left(x_{i}\right)\trianglelefteq f\left(x_{k}\right)\land f\left(x_{i}\right)\vartriangleleft\forall_{k=1}^{i-1}f\left(x_{k}\right)\label{eq:Best}
        \end{equation}

    \subsection{Markov Process}
        A function $K:\varOmega_{1}\times\Sigma_{2}\rightarrow\left[0,1\right]$, with $\left(\varOmega_{1},\Sigma_{1}\right)$ and $\left(\varOmega_{2},\Sigma_{2}\right)$ measurable spaces, is called a \textbf{(Markov) kernel} if the following two conditions hold:
        \begin{enumerate}
            \item  Function $K_{x,\bullet}\vcentcolon A\mapsto K(x,A)$ is a probability measure for each fixed $x\in\varOmega_{1}$
            \item  Function $K_{\bullet,A}\vcentcolon x\mapsto K(x,A)$ is a measurable function for each fixed $A\in\Sigma_{2}$. 
        \end{enumerate}

        Gomez considers kernels having transition densities. If the transition density $K\colon\varOmega_{1}\times\varOmega_{2}\rightarrow\left[0,1\right]$ exists, then the transition kernel can be defined using equation \ref{eq:kernel-density}. 

        \begin{equation}
            \label{eq:kernel-density}
            K\left(x,A\right)=\int_{A}K\left(x,y\right)dy
        \end{equation}

        Composition of two kernels ($K_{1}$ and $K_{2}$) is defined in terms of the kernel multiplication operator, see equation \ref{eq:multiplyKernels}. Since the kernel multiplication is an associative operator \cite{FristedtGray97}, the ordered composition of $n$ transition kernels $K_{1}$, ..., $K_{n}$ is the product kernel $K_{n}\circ K_{n-1}\circ\ldots\circ K_{1}$.
        
        \begin{equation}
            \label{eq:multiplyKernels}
            \left(K_{2}\circ K_{1}\right)\left(x,A\right)=\int K_{2}\left(y,A\right)K_{1}\left(x,dy\right)
        \end{equation}

       Finally, the probability to transit to some set $A\in\Sigma$ within $t$ steps when starting at state $x\in\varOmega$, using kernel $K$, is given by equation \ref{eq:Kernel-Iteration}. While the probability that such a Markov process is in set $A\in\Sigma$ at step $t\geq0$, when $p:\Sigma\rightarrow\left[0,1\right]$ is the initial distribution of subsets, is given by equation \ref{eq:kernel-transition2}.

        \begin{equation}
            \label{eq:Kernel-Iteration}
            K^{\left(t\right)}\left(x,A\right)=
            \begin{cases}
                K\left(x,A\right) &  \text{if } t=1\\
                {\displaystyle \intop_{\varOmega}}K^{\left(t-1\right)}\left(y,A\right)K\left(x,dy\right) &  \text{if } t>1
            \end{cases}
        \end{equation}

        \begin{equation}
            \label{eq:kernel-transition2}
            Pr\left\{ X_{t}\in A\right\} =
            \begin{cases}
                p\left(A\right) & \text{if } t=0\\
                {\displaystyle \intop_{\varOmega}}K^{\left(t\right)}\left(x,A\right)p\left(dx\right) & \text{if } t>0
            \end{cases}
        \end{equation}

    \subsection{Convergence}
        \cite{GOMEZ201953} amends the convergence approach of \cite{Rudolph96convergenceof} by defining the set of $\epsilon$-states, i.e., a set with closeness function value less than $\epsilon\in\mathbb{R}^{+}$. Let $\varOmega\subseteq\varPhi$ be a set, $f\vcentcolon\varPhi\rightarrow\mathbb{R}$ be an objective function, $\epsilon>0$ be a real number, $x\in\varOmega^{m}$, with $m$ the size of the population, and $d\left(x\right)=f\left(\mbox{\textsc{Best}}\left(x\right)\right)-f^{*}$.

        \begin{enumerate}
            \item $x$ is an $\epsilon$\textbf{-state} iff $x\in\varOmega_{\epsilon}^{m}=\left\{ x\in\varOmega^{m}\vcentcolon d\left(x\right)<\epsilon\right\} $, 
            \item $x$ is an $\overline{\epsilon}$\textbf{-state (closed)} iff $x\in\varOmega_{\overline{\epsilon}}^{m}=\left\{ x\in\varOmega^{m}\vcentcolon d\left(x\right)\leq\epsilon\right\} $,
            \item $x$ is an $\widehat{\epsilon}$\textbf{-state  (adherent)} iff $x\in\varOmega_{\widehat{\epsilon}}^{m}=\left\{ x\in\varOmega^{m}\vcentcolon d\left(x\right)=\epsilon\right\} $.
        \end{enumerate}
        
        \begin{prop}
            \label{lem:ZeroConvergence}
            Let $P_{t}\in\varOmega^{n}$ be the population maintained by an {\sgoal}. A {\sgoal} converges to the global optimum if its associated random sequence $\left(D_{t}=d\left(P_{t}\right)\vcentcolon t\geq0\right)$, converges completely to zero, i.e., equation \ref{eq:completely-0} holds for every $\epsilon>0$.
            \begin{equation}
                \label{eq:completely-0}
                \underset{t\rightarrow\infty}{\lim}{\displaystyle \sum_{i=1}^{t}Pr\left\{     \left|D_{t}\right|>\epsilon\right\} }<\infty
            \end{equation}
        \end{prop}
        
\section{Generalizing the Systematic Formal Approach to Non-Stationary \sgoals}
    For a \textbf{non-stationary (or non-homogeneous)} Markov process, the transition probabilities (kernel) may change over time (\cite{Bowerman74}). Suppose that $K_{t}$ is the transition kernel applied at time $t>0$ of a non-stationary Markov process. Then, the transition kernel of such non-stationary Markov process at time $t$ is defined as $K^{\left(t\right)}=K_{t}\circ K_{t-1}\circ\ldots\circ K_{1}$. Clearly, we can rewrite the transition kernel of a non-stationary Markov process (equation \ref{eq:NS-Kernel-Iteration}) to resemble equation \ref{eq:Kernel-Iteration}.
    
    \begin{equation}
        \label{eq:NS-Kernel-Iteration}
        K^{\left(t\right)}\left(x,A\right)=
        \begin{cases}
            K_1\left(x,A\right) &  \text{if } t=1\\
            {\displaystyle \intop_{\varOmega}}K^{\left(t-1\right)}\left(y,A\right)K_t\left(x,dy\right) &  \text{if } t>1
        \end{cases}
    \end{equation}

    Now we are in the position of generalizing Lemma 71 in \cite{GOMEZ201953} to non-stationary Markov processes.
    
    \begin{lem}
        \label{lem:(Gomez-1)}
        If for all $t\geq1$ we have that  $K_t\left(x,\varOmega_{\epsilon}\right)\geq\delta>0$ for all $x\in\Omega_{\epsilon}^{c}$ and $K_t\left(x,\varOmega_{\epsilon}\right)=1$ for all $x\in\varOmega_{\epsilon}$, then $K^{\left(t\right)}\left(x,\varOmega_{\epsilon}\right)\geq1-\left(1-\delta\right)^{t}$ holds for $t\geq1$.
    \end{lem}

    \begin{proof} 
        We just rewrite the proof of Lemma 71 in \cite{GOMEZ201953} (Gomez uses induction on $t$) but taking care of the non-stationary property of the Markov process. For $t=1$ we have that $K^{\left(t\right)}\left(x,\varOmega_{\epsilon}\right)=K_t\left(x,\varOmega_{\epsilon}\right)$ (equation \ref{eq:Kernel-Iteration}), so $K^{\left(t\right)}\left(x,\varOmega_{\epsilon}\right)\geq\delta$ (condition lemma), therefore $K^{\left(t\right)}\left(x,\varOmega_{\epsilon}\right)\geq1-\left(1-\delta\right)^{t}$ ($t=1$ and numeric operations). Here, we will use the notation (as Gomez did) $K^{\left(t\right)}\left(y,\varOmega_{\epsilon}\right)=K_{y}^{\left(t\right)}\left(\varOmega_{\epsilon}\right)$ to reduce the visual length of the equations.

        \noindent %
        \begin{tabular}{ll}
            $K_{x}^{\left(t+1\right)}\left(\varOmega_{\epsilon}\right)$ & \tabularnewline
            $={\displaystyle \intop_{\varOmega}}K_{y}^{\left(t\right)}\left(\varOmega_{\epsilon}\right)K_t\left(x,dy\right)$ & (equation \ref{eq:Kernel-Iteration})\tabularnewline
            $={\displaystyle \intop_{\varOmega_{\epsilon}}}K_{y}^{\left(t\right)}\left(\varOmega_{\epsilon}\right)K_t\left(x,dy\right)+{\displaystyle \intop_{\varOmega_{\epsilon}^{c}}}K_{y}^{\left(t\right)}\left(\varOmega_{\epsilon}\right)K_t\left(x,dy\right)$ & ($\varOmega=\varOmega_{\epsilon}\bigcup\varOmega_{\epsilon}^{c}$)\tabularnewline
            $={\displaystyle \intop_{\varOmega_{\epsilon}}}K_t\left(x,dy\right)+{\displaystyle \intop_{\varOmega_{\epsilon}^{c}}}K_{y}^{\left(t\right)}\left(\varOmega_{\epsilon}\right)K_t\left(x,dy\right)$ & (If $y\in\varOmega_{\epsilon},\,K_{y}^{\left(t\right)}\left(\varOmega_{\epsilon}\right)=1$)\tabularnewline
            $=K_t\left(x,\varOmega_{\epsilon}\right)+{\displaystyle \intop_{\varOmega_{\epsilon}^{c}}}K_{y}^{\left(t\right)}\left(\varOmega_{\epsilon}\right)K_t\left(x,dy\right)$ & (def kernel)\tabularnewline
            $\geq K_t\left(x,\varOmega_{\epsilon}\right)+\left[1-\left(1-\delta\right)^{t}\right]{\displaystyle \intop_{A_{\epsilon}^{c}}}K_t\left(x,dy\right)$ & (Induction hypothesis)\tabularnewline
            $\geq K_t\left(x,\varOmega_{\epsilon}\right)+\left[1-\left(1-\delta\right)^{t}\right]K_t\left(x,\varOmega_{\epsilon}^{c}\right)$ & (def kernel)\tabularnewline
            $\geq K_t\left(x,\varOmega_{\epsilon}\right)+K_t\left(x,\varOmega_{\epsilon}^{c}\right)-\left(1-\delta\right)^{t}K_t\left(x,\varOmega_{\epsilon}^{c}\right)$ & \tabularnewline
            $\geq1-\left(1-\delta\right)^{t}\left(1-K_t\left(x,\varOmega_{\epsilon}\right)\right)$ & (Probability)\tabularnewline
            $\geq1-\left(1-\delta\right)^{t}\left(1-\delta\right)$ & (condition lemma)\tabularnewline
            $\geq1-\left(1-\delta\right)^{t+1}$ & \tabularnewline
        \end{tabular}
        .
    \end{proof}

    Finally, Theorem 72 in \cite{GOMEZ201953} also holds for non-stationary Markov processes. So, in order to show convergence of a non-stationary {\sgoal} it is sufficient to prove that the {\sgoal} satisfies the condition of lemma \ref{lem:(Gomez-1)}.
    
    \begin{thm}
        \label{thm:Gomez-1}
        (Theorem 72 in \cite{GOMEZ201953} - a corrected version of Theorem 1 in \cite{Rudolph96convergenceof}) A {\sgoal} whose stochastic kernel satisfies $K^{\left(t\right)}\left(x,\varOmega_{\epsilon}\right)\geq1-\left(1-\delta\right)^{t}$ for all $t\geq1$ will converge to the global optimum ($f^{*}$) of a well-defined real-valued function $f:\varPhi\rightarrow\mathbb{R}$, defined in an arbitrary space $\varOmega\subseteq\varPhi$, regardless of the initial distribution $p\left(\cdot\right)$.
    \end{thm}
    \begin{proof}See proof of Theorem 72 in \cite{GOMEZ201953}.\end{proof}

\section{Selection Schemes Formalization}
    A \textbf{Selection Scheme}, is a method of selecting a group of individuals from a population (\cite{doi:10.1162/evco.1996.4.4.361}). Many schemes define an individual selection mechanism $\textsc{s1}\colon\varOmega^{\lambda}\rightarrow\varOmega$, and selects a group of individuals by repeatedly applying \textsc{s1}. In this paper, we study the uniform, fitness proportional, tournament (\cite{Miller95geneticalgorithms}), roulette, and ranking selection schemes:
    \begin{enumerate}
        \item A \textbf{uniform} scheme ($\textsc{Uniform1}\colon\varOmega^{\lambda}\rightarrow\varOmega$) gives to each candidate solution $i=1,2,\ldots,\lambda$, the same selection's probability $p\left(x_i\right)=\frac{1}{\lambda}$.
        \item A \textbf{fitness proportional} scheme ($\textsc{Proportional1}\colon\varOmega^{\lambda}\rightarrow\varOmega$) gives to each candidate solution  $i=1,2,\ldots,\lambda$, a selection's probability $p\left(x_i\right)$ such that $p\left(x_i\right) < p\left(x_j\right)$ if $f\left(x_j\right) \vartriangleleft f\left(x_i\right)$ and $p\left(x_i\right) = p\left(x_j\right)$ if $f\left(x_i\right) = f\left(x_j\right)$.
        \item A \textbf{tournament} scheme ($\textsc{Tournament1}^{m}\colon\varOmega^{\lambda}\rightarrow\varOmega$) of size $m$ chooses $m$ individuals using a \textsc{Uniform} scheme and selects an individual from these using a \textsc{Proportional1} scheme, $\textsc{Tournament1}^{m}=\textsc{Proportional1}\circ\textsc{Uniform}^{m}$ .
        \item A \textbf{roulette} scheme ($\textsc{Roulette1}\colon\varOmega^{\lambda}\rightarrow\varOmega$) is a fitness proportional one where $p\left(x_i\right)=\frac{rate(x_i)}{\sum_{i=1}^{\lambda}rate(x_i)}$ with $rate\left(x_i\right) < rate\left(x_j\right)$ if $f\left(x_j\right) \vartriangleleft f\left(x_i\right)$ and $rate\left(x_i\right) = rate\left(x_j\right)$ if $f\left(x_i\right) = f\left(x_j\right)$. If $f\left(x_i\right) \geq 0$ for all $i=1,2,\ldots,\lambda$ and maximizing then $rate\left(x_i\right)$ can be set to $f\left(x_i\right)$.
        \item A \textbf{ranking} scheme ($\textsc{Ranking1}\colon\varOmega^{\lambda}\rightarrow\varOmega$) is a roulette one with $rate\left(x_i\right)={1+|\{x_k \colon f\left(x_i\right) \vartriangleleft f\left(x_k\right) \}|}$.
    \end{enumerate}
    
    \begin{prop}
        \label{prop:S1}
        If $\textsc{s1}\colon\varOmega^{\lambda}\rightarrow\varOmega$ is a selection scheme with kernel $K_{\textsc{s1}}$ then        $\textsc{s}\colon\varOmega^{\lambda}\rightarrow\varOmega^{\mu}$ has kernel $K_{\textsc{s}}=\circledast_{i=1}^{\mu}K_{\textsc{s1}}$.
    \end{prop}
    \begin{cor}
        If $\textsc{s1}$ is based on a probability function then $K_{\textsc{s}}$ is a kernel.
    \end{cor}
    \begin{cor}
        The \textsc{Uniform, Proportional, Tournament, Roulette} and \textsc{Ranking} selection schemes have Markov kernels.
    \end{cor}
    
\section{Simulated Annealing (\textsc{sa})}
    \subsection{Concept}
        The Simulated Annealing algorithm (\textsc{sa}) considers the idea behind the process of heating and cooling a material to recrystallize it, see algorithm \ref{alg:simulated}.  When the temperature decreases, the material settles into a more ordered state, and the state into which they settle is not always the same. This state tends to have low energy compared when the material is in the presence of high temperature (\cite{Simon13}). If we consider energy as a cost function, we can use this approach to minimize cost functions. Therefore, SA is an stochastic algorithm that works with a single-individual that generates a single candidate-solution $x$ (parent)  and sets a high temperature to explore the search space. Then, some variation mechanism generates a new candidate-solution $x'$ (child) and measures its cost. A replacement policy, that fitness function and the temperature, picks one individual between the father and the child. Finally, a process decreases the temperature looking for each new solution having less energy. 

        Clearly, the replacement policy in algorithm \ref{alg:simulated} (lines 6,...,11) is not elitist. This allows \textsc{sa} to expand the search but can lead to the loss of some good candidate-solutions. In practice, it is normal to keep track of the best solution found so far\cite{Simon13}. If this is done, the replacement policy is an elitist one.

        \begin{algorithm}
            \caption{\label{alg:simulated}Simulated Annealing \cite{Simon13}}
            \begin{flushleft} \Call{Simulated annealing}{} \end{flushleft}

            \begin{algorithmic}[1]
                \State $T =$ initial temperature $>$ 0
                \State $\alpha(T) = $ cooling function: $\alpha(T) \in [0,T]$ for all $T$
                \State Initialize a candidate solution $x_{0}$ to minimization problem $f(x)$
                
                \While{\textsc{\textlnot TerminationCondition}()}
                    \State Generate a candidate solution $x$
                    \State \textbf{if} $f(x) < f(x_{0})$
                    \State \indent $x_{0} = x$
                    \State \textbf{else}
                    \State \indent $r = U[0,1]$
                    \State \indent \textbf{if} $r<exp[(f(x_{0})-f(x))/T]$
                    \State \indent \indent $x_{0} = x$
                    \State $T = \alpha(T)$
                \EndWhile
            \end{algorithmic} 
        \end{algorithm}

    \subsection{Formalization}
    
        To formalize and characterize (\textsc{sa}), we use the approach proposed by \cite{GOMEZ201953}. We rewrite algorithm \ref{alg:simulated} in terms of individual non-stationary stochastic methods, see algorithm \ref{alg:simulatedvr}. This new algorithm is in terms of \Call{Variation}{}-\Call{Replacement}{} methods. Observe that algorithms \ref{alg:simulated} and \ref{alg:simulatedvr} are equivalents. Line 5  of algorithm \ref{alg:simulated} is the method \Call{Variate$_{SA}$}{} (line 1) of algorithm \ref{alg:simulatedvr};  lines 6 to 11 of algorithm \ref{alg:simulated} is the method \Call{Replace$_{SA}$}{} (line 2) of algorithm \ref{alg:simulatedvr}. Finally, line 12 of algorithm \ref{alg:simulated} and method \Call{UpdateParameters}{} (line 3) perform the same task.

        \begin{algorithm}
            \caption{\label{alg:simulatedvr}Simulated Annealing in terms of \textbf{VR} methods}
        
            \begin{flushleft} \Call{NextPop$_{SA}$}{x} \end{flushleft}
            \begin{algorithmic}[1]
                \State $x' = $\Call{Variate$_{SA}$}{x}
                \State $x' = $\Call{Replace$_{SA_{T}}$}{x',x}
                \State \Call{UpdateParameters}{T}
                \State \Return $x'$
            \end{algorithmic} 
        \end{algorithm}

        Now, we concentrate on characterizing (\textsc{sa}) as a VR stochastic method and analyzing its convergence through non-stationary Markov kernels.

        \begin{prop}
            \label{prop:sa-1}
            If \Call{Replace$_{SA}$}{$x,x$} is an elitist method, then it can be characterized by the Markov Kernel $R_{SA} \colon \Omega^{2} \times \Sigma \longrightarrow [1,0]$ defined as 

            \begin{equation}K_{R_{SA}} = \pi_{1} \circ s_{2}\end{equation}
        
        \end{prop}

        \begin{proof}
            It is defined in the same way that the method  of \Call{R$_{HC}$}{} in \cite{GOMEZ201953}. So the proof uses the same argument that lemma 75 of \cite{GOMEZ201953}.
        \end{proof}
        
        \begin{prop}
            \label{prop:sa-2}
            if the stochastic method \Call{Variate$_{AS_{T}}$}{} can be characterized by a non-stationary Markov kernel $V_{SA_{T}}^{(t)} \colon \Omega \times \Sigma \longrightarrow [1,0]$ and condition of proposition \ref{prop:sa-1} are fulfilled then method the \Call{NextPop$_SA$}{x} can be described as a \textbf{VR} non-stationary Markov Kernel defined as

            \begin{equation}K_{SA}^{(t)} = K_{R} \circ K_{ V_{SA_{T}}}^{(t)}\end{equation}
            
        \end{prop}
        \begin{proof}
            $K_{SA}^{(t)}$ is a kernel composition under the given conditions.
        \end{proof}

        \begin{prop}
            \label{prop:sa-3}
            if \Call{Replace$_{SA}$}{} is an elitist method, then \Call{NextPop$_{SA}$}{} can be characterized by an elitist non-stationary Markov kernel.
            \end{prop}
        \begin{proof}
            This proof uses the same argument as proposition 77 in \cite{GOMEZ201953}.
        \end{proof}

    \subsection{Convergence}
        \begin{cor}
            \label{cor:sa1}
            if conditions of propositions \ref{prop:sa-1}, \ref{prop:sa-2} and \ref{prop:sa-3}, are fulfilled and method \Call{Variate$_{AS_{T}}$}{} is optimal strictly bounded from zero then \Call{NextPop$_{SA}$}{} is optimal strictly bounded from zero.
        \end{cor}
        \begin{proof}
            Follows from definition 67, lemma 68, and definition 69 in \cite{GOMEZ201953} and proposition \ref{prop:sa-3} that establish that \Call{NextPop$_{SA}$}{} can be characterized by an elitist kernel, and this  is optimal strictly bounded from zero.
        \end{proof}

        \begin{thm}
            \textsc{sa} will converge to the global optimum if \Call{Replace$_SA$}{} is elitist and if \Call{Variate$_{AS_{T}}$}{} is optimal strictly bounded from zero.
        \end{thm}
        
        \begin{proof}
            Follows from corollary \ref{cor:sa1}, and propositions \ref{prop:sa-1}, \ref{prop:sa-2} and \ref{prop:sa-3}.
        \end{proof}

\section{Evolutionary Strategies (\textsc{es})}
    \subsection{Concept}
        Evolutionary Strategies ($\mu / \rho \myreplace \lambda$)-\textsc{es} are a type of Evolutionary Algorithms that apply mutation, recombination, and selection operators to a population of individuals \cite{Beyer2002}, see algorithm \ref{alg:beyer_es}. Every individual is an \textsc{es} that has two parts: the candidate solution ($x$) and the set of endogenous strategy parameters ($s$) used to control the mutation operator (\cite{Beyer2002}). An \textsc{es} randomly initializes the population, line 2, and evolves both parts of the individual (lines 5-9) up to certain ending-condition is fulfilled (line 3). The set of endogenous parameters are exposed to evolution (lines 6 and 8) before producing a child candidate solution (line 7 and 9) to introduce variety. The new individual is a composition of a set of selected candidate solutions (line 5). \textsc{es} generates a new population of $\lambda$ new individuals each generation (line 4). Finally, \textsc{es} selects a final population using two possible approaches. The ($\mu$ + $\lambda$)-\textsc{es} approach that selects the best $\mu$ individuals among the $\mu$ parents and $\lambda$ children or   the ($\mu$,$\lambda$)-\textsc{es} that selects the best $\mu$ individuals from the $\lambda$ children (notice that $\lambda\geq\mu$ in this case). In this work, we study both of them.
    
         \begin{algorithm}
            \caption{\label{alg:beyer_es}Evolutionary strategies described by \cite{Beyer2002}}
        
            \begin{flushleft} \Call{ES}{$\mu / \rho \myreplace \lambda$}\end{flushleft}
            
            \begin{algorithmic}[1]
                \State $g=0$
                \State \Call{initialize}{$P_{q}^{(0)}\colon= \{(y_{m}^{(0)},s_{m}^{(0)},F(y_{m}^{(0)}))$,$m=1,...,\mu$\}}
            
            
                \While{\textsc{\textlnot TerminationCondition}()}
                    \For {$l = \textit{1}$ to $\lambda$ }
                        \State $a_l $ = \Call{Marriage}{$P_{q}^{g}, \rho$}
                        \State $s_l$ = \Call{Recombination$_s$}{$a_{l}$}
                        \State $y_l$ = \Call{Recombination$_y$}{$a_{l}$}
                        \State $s'_l$ = \Call{Mutation$_s$}{$s_l$} 
                        \State $y'_l$ = \Call{Mutation$_s$}{$y_l$, $s'_l$}
                        \State $F'_l$ = $F(y'_l)$ 
                    \EndFor
                    \State $P^{g}_{0} = \{ (y'_l, s'_l, F'_l), l = 1,..., \lambda \}$
                    \State \textbf{if} ($\mu, \lambda$) \textbf{then}
                    \State \indent $P_{q}^{g+1} = \Call{Selection}{P^{g}_{0}, \mu}$
                    \State \textbf{else} ($\mu + \lambda$)
                    \State \indent  $P_{q}^{g+1} = \Call{Selection}{P^{g}_{0}, P_{q}^{g}, \mu}$
                    \State g = g+1
                \EndWhile
            \end{algorithmic} 
        \end{algorithm}
    
    \subsection{Formalization}
        To formalize and characterize ($\mu / \rho \myreplace \lambda$)-\textsc{es}, we rewrite algorithm \ref{alg:beyer_es} in terms of individual non-stationary stochastic methods, see algorithm \ref{algorithm1}. This follows the approach in \cite{GOMEZ201953} that express the algorithms in terms of Variation-Replacement methods to study their convergence properties.

        Notice that algorithms \ref{alg:beyer_es} and \ref{algorithm1} are equivalents: lines 4-11 in algorithm \ref{alg:beyer_es} is method  \Call{Variate}{P}  (line 1) in the \Call{NextPop}{} method of algorithm \ref{algorithm1}. Also, lines 12-15 in algorithm \ref{alg:beyer_es}  are line 2 in the \Call{NextPop}{} method of algorithm \ref{algorithm1}. Using this characterization, we proceed to characterize each method of algorithm \ref{algorithm1} through non-stationary Markov kernels.
 
        \begin{algorithm}
            \caption{\label{algorithm1}Evolutionary strategies algorithm - NextPop method described in terms of VR methods}
        
            \begin{flushleft} \Call{NextSubPop$_i$}{P}\end{flushleft}
            \begin{algorithmic}[1]
                \State $a=$ \Call{PickParents}{$P$}
                \State $q=$ \Call{Xover$_a$}{$P$}
                \State \Call{UpdateStrategies$_a$}{s, $i$}
                \State $q'=$ \Call{Variate$_{s}$}{$q$}
                \State \Return $q'$
            \end{algorithmic} 
                 \begin{flushleft} \Call{UpdateStrategies$_a$}{$s$, $i$}\end{flushleft}
            \begin{algorithmic}[1]
                \State  $s'$  = \Call{XoverStrategie$_a$}{$s$} 
                \State  $s_i$  = \Call{VariateStrategie}{$s'$} 
            \end{algorithmic}
    
            \begin{flushleft} \Call{Variate}{P}\end{flushleft}
            \begin{algorithmic}[1]
                \For {$i = \textit{1}$ to $\lambda$ }
                    \State$Q_{i}$  = \Call{NextSubPop$_{i}$}{$P$}    
                \EndFor
                \State \Return {Q}
            \end{algorithmic}
    
            \begin{flushleft} \Call{NextPop{$_\Psi$}}{P}\end{flushleft}
            \begin{algorithmic}[1]
                \State  $Q'$ = \Call{Variate}{P}
                \State  $Q$ = \Call{Replace{$_\Psi$}}{P, $Q'$}
                \State \Return $Q$
            \end{algorithmic}
        \end{algorithm}

        With the object of characterizing ($\mu / \rho \myreplace \lambda$)-ES  we need to establish some non-stationary  Markov kernels. First, we study the \Call{Variate}{} method (line 1, method \Call{NextPop}{}, algorithm \ref{algorithm1}).
    
        Following definition 53 in \cite{GOMEZ201953}, we can express the variation method $\Call{Variate}{} \colon \Omega ^{\mu}   \longrightarrow \Omega ^{\lambda}$ as a joined stochastic method.
    
        \begin{equation}
            \begin{array}{rcl}
                \Call{Variate}{}(P) = \prod_{i=1}^{\lambda}\Call{NextSubPop}{}_{i}(P)
            \end{array}
        \end{equation}
    
        Where $\Call{NextSubPop}{} \colon \Omega ^{\mu} \longrightarrow \Omega $ chooses $\rho$ individuals from the population, combines the $\rho$ individuals, generates a child and finally mutates the strategy and the child.
    
        \begin{prop}
            \label{prop:es-1}
            If lines 8 and 9 of method \Call{UpdateStrategies}{} of algorithm \ref{algorithm1} can be characterized by non-stationary kernels $X \colon \mathbb{R}^\rho \times \mathscr{B}(\mathbb{R})^{\otimes \rho}\longrightarrow [0,1]$ and $VS^{(t)} \colon \mathbb{R} \times \mathscr{B}(\mathbb{R})\longrightarrow [0,1]$ respectively. \Call{UpdateStrategies}{} can be characterized by a non-stationary kernel $US^{(t)} \colon \mathbb{R}^\rho \times \mathscr{B}(\mathbb{R})\longrightarrow [0,1]$ defined as:
            $$K_{US}^{(t)}= K_{VS}^{(t)} \circ K_{XS}$$
        \end{prop}
    
        \begin{proof}
            It is in terms of kernel composition, follows from definition 25 of \cite{GOMEZ201953}.
        \end{proof}
    
        \begin{prop}
            \label{prop:es-2}
            If lines 2 and 4 of algorithm \ref{algorithm1} can be characterized by non-stationary Markov kernels $\Call{Xover}{}_{a} \colon ( \Omega ^{\rho}  \times \Sigma )  \longrightarrow [0,1]$ and  $\Call{Variate}{}_s \colon ( \Omega  \times \Sigma )  \longrightarrow [0,1]$ respectively,  then the method $\Call{NextSubPop}{}$ can be characterized by the kernel $\Call{NextSubPop}{} \colon ( \Omega ^{\rho}  \times \Sigma )  \longrightarrow [0,1]$ defined as the non-stationary kernel:
            $$K_{\Call{NextSubPop}{}} = K_{\Call{Variate}{}_{s}}^{(t)} \circ  K_{XOVER} \circ \pi_{\big\{ 1,..., \rho \big\}} \circ K_{\mathcal{P}}  $$ 


    
    
        \end{prop}
        \begin{proof}
            It is in terms of kernel composition, follows from definition 25 of \cite{GOMEZ201953}.
        \end{proof}
    
        \begin{prop}
            \label{prop:es-3}
            If $\Call{NextSubPop}{}$ can be characterized by a non-stationary Markov kernel, the stochastic method $\Call{Variate}{}^{(t)}$ can be characterized by a kernel $V \colon \Omega ^{\mu} \times \Sigma ^{\otimes \mu} \longrightarrow [0,1] $ defined as   
            $$K_{\Call{Variate}{}}^{(t)} = [\varoast_{i=1}^{\lambda}[K_{\Call{NextSubPop}{}_{i}}]]$$
        \end{prop}
    
        \begin{proof}
            It is a join stochastic method, follows from definition 55 and proposition 56 of \cite{GOMEZ201953}.
        \end{proof}
    
        \begin{prop}
            \label{prop:es-4}
            The stochastic method \Call{Replace}$_{(\mu+\lambda)}$ used in line 2 of method \Call{NextPop}{}, can be characterized by the kernel $R_{\mu,\mu + \lambda} \colon \Omega ^{\mu+\lambda}   \times \Sigma^{\otimes \mu}  \longrightarrow [0,1] $ defined as
            $K_{R_{\mu,\mu + \lambda}} = \pi_{\big\{ 1,...,\mu \big\} } \circ s_{\mu+ \lambda,\mu + \lambda -1}$ and the stochastic method \linebreak $\Call{Replace}{}_{(\mu,\lambda)}$, can be characterized by the kernel $R_{\mu, \lambda} \colon \Omega ^{\lambda}   \times \Sigma^{\otimes \mu}  \longrightarrow [0,1$ defined as
            $K_{R_{\mu,\lambda}} = \pi_{\big\{ 1,...,\mu \big\} } \circ s_{\lambda,\lambda -1}$ 
        \end{prop}

        \begin{proof}
            $K_{R_{\mu,\lambda}}$ and $K_{R_{\mu+\lambda}}$ are kernels composition. Follows from definition 25 in \cite{GOMEZ201953}.
        \end{proof}
        \begin{cor}
            If methods \Call{PickParents}{}, \Call{XOver$_a$}{}, \Call{XoverStrategie$_a$}{}, \Call{VariateStrategie}{} and, \Call{Variate$_{s}$}{} can be described by Markov kernels  fulfilling the  conditions of \ref{prop:es-1} and \ref{prop:es-2},  evolutionary Strategies can be described by a \textbf{VR} kernel.
    	    $$K_{ES} = K_{R} \circ K_{V}$$
    	
    	    where:
    	
            $$K_{V}=K_{\Call{Variate}{}}$$
    	
    	
        	\begin{equation}
                K_{R} = K_{R_{\mu,\lambda}}
                \text{ or }
                K_{R} = K_{R_{\mu+\lambda}}
            \end{equation}
    
        \end{cor}
        \begin{proof}
            Follows from propositions \ref{prop:es-1}, \ref{prop:es-2}, \ref{prop:es-3} and \ref{prop:es-4}.
        \end{proof}

    \subsection{Convergence}
    
        \begin{prop}
            \label{prop:elitist-es}
            The \Call{NextPop{$_{(\mu / \rho + \lambda)-ES}$}}{} is an elitist stochastic method that can be characterized by an elitist stochastic kernel
        \end{prop}
        
        \begin{proof}
            Let $k \in  [1,\mu]$ be the index of the best individual in population $P$, then $f(BEST(P)) = f(P_{k})$. Since $P \subseteq \{P \cup \Call{Variate}{P}\}$ and the method \Call{Replace}{} is elitist. It is clear that $f(BEST(P \cup  \Call{Variate}{P})) \eqtriangle f(P_{k}) $. 
        \end{proof}
        
        
        \begin{cor}
            \label{cor:es-1}
            If conditions of proposition \ref{prop:es-1} and \ref{prop:es-2} are satisfied and \Call{Variate$_{s}$}{}  is optimal strictly bounded from zero then the method \Call{NextPop}{}$_{\mu + \lambda}$ is optimal strictly from zero.
        \end{cor}
        \begin{proof}
            Follows from definition 67, lemma 68, and definition 69 of \cite{GOMEZ201953} and proposition \ref{prop:elitist-es} that establish that an elitist kernel is optimal strictly bounded from zero.
        \end{proof}
        
        \begin{thm}
            \label{thm:es-convergence}
            ($\mu / \rho + \lambda$)-ES will converge to the global optimum if methods \Call{PickParents}{} and \Call{Variate}{}$_{s}^{(t)}$ can be characterized by stationary or non-stationary Markov kernels and  \Call{Variate$_{s}$}{} is optimal strictly bounded from zero.
        \end{thm}
        \begin{proof}
            Follows from theorem \ref{thm:Gomez-1} and  corollary \ref{cor:es-1}.
        \end{proof}

\section{Conclusion}

In this paper we have generalized the conditions of convergence to the global optimum from stationary to non-stationary Markov process that are present in the work of stochastic global optimization algorithms: a systematic approach proposed by \cite{GOMEZ201953}. We formalize some selection schemes to generalize the theory to cover as many variations of each algorithm as possible. Also, we formalized and characterized both simulated-annealing and evolutionary-strategies using the developed theory. There, we established which conditions must be fulfilled to achieve a global convergence in both algorithms. Our future work will concentrate on using the proposed approach to formalize as many stationary and non-stationary \textsc{SGoal} as possible, and  extending and developing the theory for several particular methods that can be considered \textsc{SGoal}s.

\bibliography{references}
\end{document}